\RequirePackage{fix-cm}
\documentclass[smallextended]{svjour3a}       % onecolumn (second format)
\smartqed  % flush right qed marks, e.g. at end of proof
\usepackage{graphicx}
\usepackage{hyperref}
\usepackage{amssymb}
\usepackage{amsmath}

\usepackage{mathptmx}      % use Times fonts if available on your TeX system
\begin{document}

\title{Kripke Semantics for Fuzzy Logics\thanks{This a part of the Ph.D. thesis of the first author written under the supervision of the second author at the University of Tabriz.}
}


\author{Parvin Safari          \and
        Saeed Salehi %etc.
}


\institute{P. Safari \at
Department of Mathematics, University of Tabriz, 29 Bahman Blvd., P.O.Box~51666-17766, Tabriz, {\sc Iran}. \quad  \email{p$_-$safari@tabrizu.ac.ir}           %  \\
           \and
          S. Salehi \at
              Department of Mathematics, University of Tabriz, 29 Bahman Blvd., P.O.Box~51666-17766, Tabriz, {\sc Iran}.   \quad   \email{salehipour@tabrizu.ac.ir} \quad
              Web: \url{http://saeedsalehi.ir}
}

\date{10 December 2016}

\maketitle

\begin{abstract}
Kripke frames (and models)  provide a suitable semantics for sub-classical logics; for example Intuitionistic Logic (of Brouwer and Heyting) axiomatizes the reflexive and transitive Kripke frames (with persistent satisfaction relations), and the Basic Logic (of Visser) axiomatizes transitive Kripke frames (with persistent satisfaction relations). Here, we investigate whether Kripke frames/models could provide a semantics for fuzzy logics. For each axiom of the Basic Fuzzy Logic, necessary and sufficient conditions are sought  for Kripke frames/models which satisfy them. It turns out that the only fuzzy logics (logics containing the Basic Fuzzy Logic) which are sound and complete with respect to a class of Kripke frames/models are the extensions of the G\"{o}del Logic (or the super-intuitionistic logic of Dummett); indeed this logic is sound and strongly complete with respect to reflexive, transitive and connected (linear) Kripke frames (with persistent satisfaction relations). This provides a semantic characterization for the G\"{o}del Logic among (propositional) fuzzy logics.

\end{abstract}

%\paragraph{Paragraph headings} Use paragraph headings as needed.
%
%\begin{equation}
%\end{equation}

\newpage 

\section{Introduction and Preliminaries}
Kripke frames provide a semantics for modal logics and for some
sub-classical logics such as Intuitionistic logic (of Brouwer and
Heyting) and Basic Logic (of Visser). Visser Basic Logic is sound
and strongly complete with respect to transitive Kripke frames \cite{Visser}
and the Intuitionistic Logic is sound and strongly complete with
respect to reflexive and transitive Kripke frames \cite{Mints}. It could be
expected that a class of Kripke frames could provide a suitable
semantics for the Basic Fuzzy Logic (introduced in \cite{Hajek}).
For each axiom of this logic, all  the  Kripke frames/models that satisfy it will be investigated.
We
shall see that the only (fuzzy) logics which contain the Basic Fuzzy Logic  and are sound and strongly complete with
respect to a class of Kripke frames/models are extensions of the  G\"{o}del Logic, or
equivalently the Dummet Logic (cf. \cite{Dummett} and \cite{Bendova}). This logic can be
aximatized as the Intuitionistic Logic plus the  axiom
 $(\varphi\rightarrow\psi)\vee(\psi\rightarrow\varphi)$, and is
 sound and strongly complete with respect to reflexive, transitive,
 and connected Kripke frames (with persistent satisfaction relations).

\begin{definition}[Kripke Frames]\label{def-frame}
A {\em Kripke frame} is a directed graph, i.e., an ordered pair $\langle K,R\rangle$ where
$R\subseteq K^{2}$ is a binary relation on $K$.
In a Kripke frame $\langle K,R\rangle$ the members of
$K$ are called {\em nodes},  and the relation $R$ is called the  {\em
accessability} relation; if $kRk'$ then the node $k'$ is said
to be {\em accessible} from the node $k$.
 \hfill $\bigtriangleup\hspace{-0.7785em}\wedge$
 \end{definition}

\begin{definition}[Reflexivity \& Transitivity]\label{def-rtl}
A relation $R\subseteq K\times K$ is
\begin{itemize}
\item {\em reflexive}, when for any $ k\!\in\!K$, $kRk$ holds.
\item {\em transitive}, when for any $k,k',k''\!\!\in\!K$, if $kRk'$ and $k'Rk''$
hold then $kRk''$ holds.
\end{itemize}
\noindent A Kripke frame is called reflexive/transitive,  when the relation $R$ is so.
 \hfill $\bigtriangleup\hspace{-0.7785em}\wedge$
 \end{definition}

\begin{definition}[Transitive Closure]\label{def-rtclosed}
For  a binary relation $R\subseteq K\times K$ on $K$  and   a node $k\!\in\!K$, let
  $R^1[k]=R[k]=\{x\!\in\!K\mid kRx\}$ be   {\em the image of $\{k\}$ under $R$},  and let
  $R^2[k]=\{x\!\in\!K\mid \exists y\!\in\!K (kRyRx)\}$, and generally for any $n\!\in\!\mathbb{N}$  let the set
 $R^n[k]$ be defined by  $\{x\!\in\!K\mid \exists y_1,\ldots,y_{n-1}\!\in\!K (kRy_1Ry_2R\cdots Ry_{n-1}Rx)\}$. The {\em transitive closure} of $R$ on $\{k\}$ is then
 $R^+[k]=\bigcup_{n=1}^\infty R^n[k]$. Define also  $R^{++}[k]=\bigcup_{n=2}^\infty R^n[k]$.
 \hfill $\bigtriangleup\hspace{-0.7785em}\wedge$
 \end{definition}

\begin{definition}[Connectedness]\label{def-conn}
 A relation $R\subseteq K\times K$ is called {\em connected}, when for any $k\!\in\!K$ and any $k',k''\!\in\!R^+[k]$, either $k'Rk''$ or  $k''Rk'$  holds (cf. \cite{svejdar}).
  \hfill $\bigtriangleup\hspace{-0.775em}\wedge$
 \end{definition}

\begin{definition}[Syntax of Fuzzy Logic]\label{def-syntx}
Formulas of {\em Propositional Fuzzy Logic} are built from the constant $\underline{\bot\!\!\!\!\!\!\bot}$ (for the falsity) and the connectives $\&,\rightarrow$ (for conjunction and implication) together with a countably infinite set of atoms, denoted $\textsf{Atoms}$.
 \hfill $\bigtriangleup\hspace{-0.7785em}\wedge$
 \end{definition}

Let us note that then the negation of a formula $\varphi$ becomes $\varphi\rightarrow\underline{\bot\!\!\!\!\!\!\bot}$ in this language.

\begin{definition}[Kripke Models]\label{def-modele}
 A {\em Kripke model} is a triple $\mathcal{K}=\langle K,R,\vDash\rangle$ where $\langle K,R\rangle$
is a Kripke frame and $\vDash\,\subseteq K\times\textsf{Atoms}$
is a {\em satisfaction} relation.
The satisfaction
relation  can be extended to all the (propositional) formulas,
i.e., to $\vDash\;\subseteq K\times\textsf{Formulas}$, as follows (\textsf{Formulas} is the set of all formulas):
\begin{itemize}
\item No node satisfies $\underline{\bot\!\!\!\!\!\!\bot}$, i.e.,  $k\nvDash\underline{\bot\!\!\!\!\!\!\bot}$ for all
    $k\!\in\!K$.
\item The conjunction is satisfied if and only if each
 component is satisfied, i.e., \\
$k\vDash (\varphi\&\psi) \iff k\vDash\varphi \textrm{ and }
k\vDash\psi$.
\item The implication is satisfied if and only if whenever
    an accessible node satisfies the antecedent then it also
    satisfies the consequent, i.e., \\
    $k\vDash(\varphi\rightarrow\psi) \iff \textrm{ for all }
    k'\!\in\!K \big(\textrm{if }k\,R\,k' \textrm{ and }
     k'\vDash\varphi \textrm{ then } k'\vDash\psi\big)$

   \hspace{5.15em}   $\iff \forall k'\!\in\!R[k]\big(k'\vDash\varphi\longrightarrow k'\vDash\psi\big)$.
 \hfill $\bigtriangleup\hspace{-0.7785em}\wedge$
\end{itemize}
 \end{definition}

\begin{remark}[Truth]\label{remark}
The formula $\underline{\bot\!\!\!\!\!\!\bot}
\rightarrow\underline{\bot\!\!\!\!\!\!\bot}$ is always true, and holds in every node of any Kripke model (by definition). Let us denote it by $\overline{\top\!\!\!\!\!\!\top}$($=\underline{\bot\!\!\!\!\!\!\bot}
\rightarrow\underline{\bot\!\!\!\!\!\!\bot}$).
 \hfill $\bigtriangleup\hspace{-0.7785em}\wedge$
\end{remark}

\begin{definition}[Satisfaction]\label{def-sat}
A formula is {\em satisfied}  in a Kripke model when it is
satisfied in every node of that model. A Kripke {\em frame satisfies} a formula when every Kripke model with that frame satisfies the
formula.
A rule is said to be satisfied in a Kripke model when the
satisfaction of the premise(s)  of the rule in a node implies the satisfaction of its conclusion in that node.
A rule is said to be satisfied in a Kripke frame when
it is satisfied in every Kripke model with that frame.
 \hfill $\bigtriangleup\hspace{-0.7785em}\wedge$
 \end{definition}

\begin{definition}[Persistency]\label{def-mono}
A satisfaction relation $\vDash \, \subseteq K\times\textsf{Atoms}$ is called to be  {\em (atom) persistent with respect to $R\subseteq K\times K$} (cf.~\cite{svejdar}) when for any $k, k'\!\in\!K$ and $p\!\in\!\textsf{Atoms}$, if $k\vDash p$ and $kRk'$
then $k'\vDash p$; this property is called {\em atom persistency}. A satisfaction relation $\vDash \, \subseteq K\times\textsf{Formulas}$ is called to be {\em (formula) persistent  with respect to $R\subseteq K\times K$}
 when for any $k, k'\!\in\!K$ and $\varphi\!\in\!\textsf{Formulas}$, if $k\vDash \varphi$ and $kRk'$ then $k'\vDash\varphi$;
this property is called {\em formula persistency}.
 \hfill $\bigtriangleup\hspace{-0.7785em}\wedge$
\end{definition}

%\begin{definition}[Satisfaction]\label{def-sat}
% \hfill$+\hspace{-0.91em}\times\hspace{-0.78em}\circ$
% \end{definition}
\paragraph{Convention.}
 The restriction of a relation $S\subseteq A\times B$ to a subset $C\subseteq A$ is denoted by $S|_C$, i.e., $S|_C=S\cap (C\times B)$.   \hfill $\bigtriangleup\hspace{-0.7785em}\wedge$

\begin{proposition}[Atom / Formula Persistency]\label{lem-per}
In a Kripke model $\langle K,R,\vDash\rangle$ if the restriction of $R$ to $R^+[k]$, i.e., $R|_{R^+[k]}$, is transitive for some node $k\!\in\!K$, then the atom persistency in (every node of) $R^+[k]$ implies the formula persistency \textup{(}in $R^+[k]$\textup{)}.
\end{proposition}

\begin{proof}
 By induction on the formula $\varphi$ we show that for every $k',k''\!\in\!R^+[k]$ if $k'Rk''$ and $k'\vDash\varphi$ then $k''\vDash\varphi$:
 \begin{itemize}
\item For atomic formula $\varphi$, we have $k''\vDash\varphi$ by the assumption (also by definition, $k''\nvDash\underline{\bot\!\!\!\!\!\!\bot}$ always holds).
\item For $\varphi=\psi\&\theta$  ($\psi$ and $\theta$ are formulas)  by definition, $k'\vDash \psi$ and $k'\vDash \theta$, so by the induction hypothesis
$k''\vDash \psi$ and $k''\vDash \theta$, whence, $k''\vDash \psi\& \theta$ holds.
\item For $\varphi=\psi\rightarrow \theta$, we show that $k''\vDash \psi\rightarrow\theta$
  which is equivalent to  $$\forall
k'''\!\in\!R[k'']\,\big(k'''\vDash \psi\Longrightarrow k'''\vDash
\theta\big).$$  So, let us assume that (1)~$R|_{R^+[k]}$ is transitive, (2)~$k'\vDash\psi\rightarrow \theta$, (3)~$k'''\vDash \psi$, and  (4)~$k' R k'' R K'''$ for $k',k'',k'''\!\in\!R^+[k]$.  By (1) and (4), we have $k'Rk'''$, and so by (2) and (3), $k'''\vDash \theta$ holds. \hfill  $\maltese\hspace{-0.805em}\Box$
\end{itemize}
 \end{proof}

\begin{lemma}[Transitivity Lemma]\label{lem-trans}
 In a Kripke frame $\langle K,R\rangle$, if $R$ is
 reflexive  and for all $k\in K$, the restriction of $R$ to $R^{+}[k]$, i.e.,  $R|_{R^{+}[k]}$, is transitive,  then $R$ is transitive.
% \textup{(}The same holds if we replace $R^+[k]$ everywhere with $R^{++}[k]$.\textup{)}
\end{lemma}

\begin{proof}
If  $R$ were not transitive,  there would exist some $k_1,k_2,k_3\!\in\!K$ such that
  $k_1 R k_2$ and $k_2 R k_3$ but $k_1\!\!\!\not\!R k_3$.
  Now, trivially, $k_2,k_3\!\in\!R^{+}[k_1]$ and by the reflexivity of $R$ we also have $k_1\!\in\!R^{+}[k_1]$. But then $R|_{R^{+}[k_1]}$ is not transitive, contradiction!
\hfill $\maltese\hspace{-0.805em}\Box$ \end{proof}

%%%%%%%%%%%%%%%%%%%%%%%%%%%%%%
%%%%%%%%%%%%%%%%%%%%%%%%%%%%%%
%%%%%%%

%%%%%%%%%%%%%%%%%%%%%%%%%%%%%%
%%%%%%%%%%%%%%%%%%%%%%%%%%%%%%
%%%%%%%

\subsection{The Basic Fuzzy Logic}
The axiom of Basic Logic (BL) are (cf.~\cite{Hajek})

\begin{itemize}
\itemindent=2em
\item[$(A_1)$]
    $(\varphi\rightarrow\psi)\rightarrow[(\psi\rightarrow\theta)\rightarrow(\varphi\rightarrow\theta)]$
\item[$(A_2)$] $(\varphi\&\psi)\rightarrow\varphi$
\item[$(A_3)$] $(\varphi\&\psi)\rightarrow(\psi\&\varphi)$
\item[$(A_4)$]
    $(\varphi\&[\varphi\rightarrow\psi])\rightarrow(\psi\&[\psi\rightarrow\varphi])$
\item[$(A_5a)$]
    $[\varphi\rightarrow(\psi\rightarrow\theta)]\rightarrow[(\varphi\&\psi)\rightarrow\theta]$
\item[$(A_5b)$]
    $[(\varphi\&\psi)\rightarrow\theta]\rightarrow[\varphi\rightarrow(\psi\rightarrow\theta)]$
\item[$(A_6)$]
    $[(\varphi\rightarrow\psi)\rightarrow\theta]\rightarrow[([\psi\rightarrow\varphi]\rightarrow\theta)\rightarrow\theta]$
\item[$(A_7)$] $\underline{\bot\!\!\!\!\!\!\bot}   \rightarrow\varphi$
\end{itemize}

\noindent and its (only) rule is Modus Ponens

\begin{itemize}
\itemindent=2em
\item[($M\!P$)] $$\frac{A, \quad A\rightarrow
    B}{B}.$$
\end{itemize}

%%%%%%%%%%%%%%%%%%%%%%%%%%%%%%
%%%%%%%%%%%%%%%%%%%%%%%%%%%%%%
%%%%%%%

\section{Basic Fuzzy Logic and Kripke Frames/Models}
It immediately follows from the definitions that
\begin{proposition}[Universality of ${\bf A_2,A_3,A_7}$, and $\varphi\rightarrow\varphi
\&\varphi$]\label{trivial}
The axioms $(A_2), (A_3)$, $(A_7)$, and also the formula
$\varphi\rightarrow(\varphi\&\varphi)$ are satisfied in every
Kripke frame.\hfill $\maltese\hspace{-0.805em}\Box$
\end{proposition}

It can also be easily checked that the Modus Ponens $(M\!P)$ rule is satisfied in every reflexive Kripke frame. The converse is also true (cf. \cite[Proposition~5.1]{CJ}).

\begin{theorem}[${\bf MP}$ \& Reflexivity]\label{reflexive1}
The only rule of the Basic Fuzzy Logic ($M\!P$) is satisfied  in a Kripke
frame $\langle K,R\rangle$ if and only if  $R$ is reflexive.
\end{theorem}
\begin{proof}
If $R$ is reflexive then for any $k\!\in\!K$ we have  $k\!\vDash\!\varphi,\;  k\!\vDash\!\varphi\!\rightarrow\!\psi\Longrightarrow k\!\vDash\!\psi$ just because  $kRk$. Now, if the relation $R$ is not reflexive then
there exists some  $k\!\in\!K$ such that $k\!\!\!\not\!R k$. For atoms $p,q$ let $\vDash$ be $\big(K\times\{p\}\big)\cup\big(R[k]\times\{q\}\big)$. Then $k\vDash p$ and $k\vDash p\rightarrow q$ because for any $k'$ with $kRk'$ we have $k'\vDash q$. But $k\nvDash q$ because $k\!\not\in\!R[k]$.  So, the rule $(M\!P)$ is not satisfied at node $k$.
\hfill $\maltese\hspace{-0.805em}\Box$ \end{proof}

The axiom $(A_1)$ is satisfied in every transitive Kripke frame. The following theorem characterizes exactly the frames in which this axiom is satisfied.

\begin{theorem}[${\bf A_1}$ \& Transitivity]\label{transitive1}
The axiom $(A_1)$ is satisfied  in a Kripke frame $\langle K,R\rangle$ if and only if
$R|_{R^{+}[k]}$  is transitive for all $k\!\in\!K$.
\end{theorem}
\begin{proof} Fix a $k\!\in\!K$ and suppose that
  $R|_{R^{+}[k]}$ is transitive. We show that  $k\vDash(A_1)$, or equivalently
  $\forall k'\!\in\!R[k]\,\big(k'\vDash(\varphi\rightarrow\psi)
  \Longrightarrow k'\vDash[(\psi\rightarrow\theta)
  \rightarrow(\varphi\rightarrow\theta)]\big).$ That is equivalent to showing, for a fixed $k'\!\in\!R[k]$,  that  $\forall k''\!\in\!R[k']\,\big(k''\vDash\psi\rightarrow\theta\Longrightarrow k''\vDash\varphi\rightarrow\theta\big),$
assuming  $k'\vDash(\varphi\rightarrow\psi)$, and
this   is in turn equivalent to showing, assuming  $k''\vDash\psi\rightarrow\theta$ for a fixed $k''\!\in\!R[k']$, that     $\forall k'''\!\in\!R[k''] \big(k'''\vDash\varphi\Longrightarrow  k'''\vDash\theta\big)$.  Thus, let us assume that
(1)~ $R|_{R^{+}[k]}$ is transitive and $kRk'Rk''Rk'''$,
(2)~$k'\vDash\varphi\rightarrow\psi$,
(3)~$k''\vDash\psi\rightarrow\theta$, and
(4)~$ k'''\vDash\varphi$.
We then show that $k'''\vDash\theta$:
By (1), since  $k''',k'',k'\!\in\!R^{+}[k]$,  we have $k'Rk'''$ and so by (2) and (4) we can infer that  $k'''\vDash\psi$. Whence, (3) implies that  $k'''\vDash\theta$ holds.

So, the {\em if} part of the theorem has been proved. For the {\em only if} part, assume  that  for a node $k_0\!\in\!K$, in a Kripke frame $\langle K,R\rangle$, the relation $R|_{R^{+}[k_0]}$ is not transitive; i.e., there are   $k_1,k_2, k_3\!\in\!R^{+}[k_0]$ such that $k_1Rk_2R k_3$
but $k_1\!\!\!\!\not\!Rk_3$.  For atoms $p,q,r$ let the satisfaction relation $\vDash$ be  $\big(K\!\times\!\{p\}\big)\cup\big(R[k_1]\!\times\!\{q\}\big)
\cup\big((R[k_1]\cap R[k_2])\!\times\!\{r\}\big).$
Since we have $k_1 ,k_2, k_3\!\in\!R^{+}[k_0]$ then there are $\ell_1,\cdots,\ell_n\!\in\!K$ (for some $n\!\geqslant\!0$) such that $k_0R\ell_1R\cdots R\ell_nRk_1Rk_2Rk_3$ (when $n\!=\!0$ then $\ell_n=k_0$).
We now show that  the instance $(p\rightarrow q)\rightarrow[(q\rightarrow
 r)\rightarrow(p\rightarrow r)]$ of $(A_1)$ is not satisfied at  $\ell_n$. To see this, we  note that $k_2\nvDash p\rightarrow r$, because $k_2Rk_3, k_3\vDash p$ but $k_3\nvDash r$ for $k_3\!\!\not\in\!\!R[k_1]$,  and
 $k_2\vDash q\rightarrow r$ because  for any $k\!\in\!K$ if $k_2Rk\vDash q$   then $k\!\in\!R[k_2]$ and $k\!\in\!R[k_1]$  so $k\vDash r$. Hence, we conclude that  $k_1\nvDash (q\rightarrow\! r)\rightarrow(p\rightarrow r)$, but $k_1\vDash p\rightarrow q$ because for any $k\!\in\!K$ if $k_1Rk\vDash p$ then $k\!\in\!R[k_1]$ and so $k\vDash q$. Thus,    $\ell_n\nvDash (p\rightarrow q)\rightarrow[(q\rightarrow
 r)\rightarrow(p\rightarrow r)]$.
\hfill $\maltese\hspace{-0.805em}\Box$ \end{proof}

%\begin{remark}
%how it goes
%\hfill$+\hspace{-0.91em}\times\hspace{-0.78em}\circ$
% \end{remark}

It can be seen that the axiom $(A_4)$ is satisfied in every reflexive Kripke model whose satisfaction relation is (formula)  persistent (with respect to the accessibility relation). Here, we give an exact characterizations for all the Kripke models  which satisfy this axiom.

\begin{theorem}[${\bf A_4}$  \&  Reflexivity+Persistency]\label{reflexivety2}
The axiom $(A_4)$ is satisfied  in every Kripke model  $\langle
K,R,\vDash\rangle$ in which for every  $k\!\in\!K$ the restricted relation
  $R|_{R^{+}[k]}$ is reflexive  and  $\vDash|_{R^{+}[k]}$ is formula persistent with respect to $R$.
  Conversely, if $(A_4)$ is satisfied in a Kripke frame then
  for all $k\!\in\!K$ the relation $R|_{R^{+}[k]}$ is reflexive and the restriction of the satisfaction relations to the sets $R^{+}[k]$ (for every $k\!\in\!K$) on those frames  should be
  formula persistent with respect to $R$.
\end{theorem}
\begin{proof}
For a fixed Kripke model $\langle K,R,\vDash \rangle$ and fixed node $k\!\in\!K$, suppose that $R|_{R^{+}[k]}$ is reflexive and that $\vDash|_{R^{+}[k]}$ has the formula persistency  property. We show that $k\vDash(A_4)$ or equivalently
$\forall k'\!\in\!R[k]\big(k'\vDash\varphi\&[\varphi\rightarrow\psi]
\Longrightarrow
k'\vDash\psi\&[\psi\rightarrow\varphi]\big).$ Thus, it suffices to show that
$k'\vDash\psi$ and $\forall k''\!\in\!R[k']\big(k''\vDash\psi\Longrightarrow k''\vDash\varphi\big),$ if $kRk'\vDash\varphi\&[\varphi\rightarrow\psi].$ Whence, we
assume that
(1)~$R|_{R^{+}[k]}$ is reflexive and $kRk'Rk''$,
 (2)~$k'\vDash\varphi\&[\varphi\rightarrow\psi]$,
(3)~$k''\vDash\psi$, and (4)~$\vDash|_{R^{+}[k]}$ is formula persistent;
 and show that $k'\vDash\psi$ and $k''\vDash\varphi.$
By (2)  we have (5)~$k'\vDash\varphi$ and
(6)~$k'\vDash\varphi\rightarrow\psi$. So, by  (4) and (1) we also have
$k''\vDash\varphi$. By (1) again, we have $k' R k'$ which by (5) and (6) implies that   $k'\vDash\psi$ holds.

Now,  we suppose that the axiom $(A_4)$ is satisfied in a Kripke frame $\langle K,R \rangle$, and show that for any $k\!\in\!K$ the relation $R|_{R^{+}[k]}$ is reflexive.  If $R|_{R^{+}[k_0]}$ is not reflexive for some $k_0\!\in\!K$, then there are $\ell_1,\cdots,\ell_n\!\in\!K$ ($n\!\geqslant\!0$) such that $k_0R\ell_1R\cdots R\ell_nRk_1\!\!\!\not\!Rk_1$. Define the  satisfaction relation $\vDash$ to be  $\langle k_1,p\rangle$ for some atom $p$. We show that under this satisfaction relation the instance
$(p\&[p\rightarrow q])\rightarrow(q\&[q\rightarrow p])$ of
$(A_4)$ is not satisfied at $\ell_n$. That is because $k_1\vDash p\&(p\rightarrow q)$ by definition and the fact that for no $k\!\in\!R[k_1]$ we can have $k\vDash p$ (by $k_1\!\!\!\not\!Rk_1$). On the other hand by definition $k_1\nvDash q$ and so $k_1\nvDash q\&(q\rightarrow p)$.

Next, if  $\vDash|_{R^{+}[k_0]}$ is not formula persistent with respect to $R$ in a Kripke model $\langle K,R,\vDash \rangle$ and node $k_0\!\in\!K$, then there are two nodes $k_1,k_2\!\in\!R^+[k_0]$ and a formula $\varphi$   such that $k_1Rk_2$ and $k_1\vDash\varphi$ but $k_2\nvDash\varphi$. Also there are some $\ell_1,\cdots,\ell_n\!\in\!K$ ($n\!\geqslant\!0$) such that $k_0R\ell_1R\cdots R\ell_nRk_1$.
We show that the instance $(\varphi\&[\varphi\rightarrow \overline{\top\!\!\!\!\!\!\top}])
\rightarrow(\overline{\top\!\!\!\!\!\!\top}
\&[\overline{\top\!\!\!\!\!\!\top}\rightarrow\varphi])$  of $(A_4)$ (see Remark~\ref{remark}  for the definition of $\overline{\top\!\!\!\!\!\!\top}$) is not satisfied in $\langle
K,R,\vDash\rangle$ at $\ell_n$:  Because, at  $k_1$ (for which $\ell_nRk_1$ holds) we have $k_1\vDash\varphi\&[\varphi\rightarrow \overline{\top\!\!\!\!\!\!\top}]$ (since $k\vDash\overline{\top\!\!\!\!\!\!\top}$  holds for any $k$) but $k_1\nvDash\overline{\top\!\!\!\!\!\!\top}\rightarrow\varphi$ since for $k_2\!\in\!R[k_1]$ we have $k_2\nvDash\varphi$ (and of course $k_2\vDash\overline{\top\!\!\!\!\!\!\top}$).
\hfill $\maltese\hspace{-0.805em}\Box$ \end{proof}

The axiom $(A_5a)$, too, is satisfied in every reflexive frame. Here is an exact characterization.

 \begin{theorem}[${\bf A_5a}$ \& Reflexivity]\label{reflexive2}
The axiom $(A_5a)$  is satisfied  in a Kripke frame $\langle
K,R\rangle$ if and only if $R|_{R^{2}[k]}$ is reflexive for all
$k\!\in\!K$.
\end{theorem}
\begin{proof}
Fix a $k\!\in\!K$ in a Kripke frame $\langle K,R \rangle$ for which $R|_{R^2[k]}$ is reflexive. For showing $k\vDash(A_5a)$, we  show that $\forall k'\!\in\!R[k] \big(k'\vDash\varphi\rightarrow(\psi\rightarrow\theta)
\Longrightarrow
k'\vDash(\varphi\&\psi)\rightarrow\theta\big)$, which is equivalent to showing $\forall k''\!\in\!R[k']\,\big(k''\vDash(\varphi\&\psi)\Longrightarrow
k''\vDash\theta\big),$ for some fixed $k'\!\in\!R[k]$ with $k'\vDash\varphi\rightarrow(\psi\rightarrow\theta).$  Whence, we assume that (1)~the relation $R|_{R^{2}[k]}$ is reflexive,  (2)~$k'\vDash\varphi\rightarrow(\psi\rightarrow\theta)$, (3)~$k''\vDash\varphi\&\psi$ and (4)~$k R k' R k''$, and show that $k''\vDash \theta$: By (3) we have (5)~$k''\vDash\psi$; the assumptions (2) and (4) imply that
(6)~$k''\vDash\psi\rightarrow\theta$.
By the reflexivity of $R|_{R^2[k]}$ and $k''\!\in\!R^2[k]$ we have $k''Rk''$, and so it follows from (5) and (6) that  $k''\vDash\theta$ holds. This proves the {\em if} part of the theorem.

For the converse, the {\em only if} part, assume that for a node $k_0\!\in\!K$ in a Kripke frame
$\langle K,R \rangle$, the restricted relation $R|_{R^{2}[k_0]}$ is not reflexive; i.e., there is $k\!\in\!{R^{2}[k_0]}$ such that
$k\!\!\!\not\!R k$. Let us note that for some $k'$ we have $k_0Rk'Rk$. Let $\vDash$ be $\big(\{k\}\times\{p,q\}\big)$ for atoms $p,q,r$. We show that the instance $[p\rightarrow(q\rightarrow r)]\rightarrow[(p\& q)\rightarrow
r]$ of $(A_5a)$ is not satisfied at $k_0$: we have $k'\nvDash(p\& q)\rightarrow r$ because
at $k\!\in\!R[k']$ we have  $k\vDash p\& q$  but $k\nvDash r$. On the other hand $k'\vDash p\rightarrow(q\rightarrow r)$
because for any
$\ell\!\in\!R[k']$  if $\ell\vDash p$ then $\ell=k$ but then $k\vDash q\rightarrow r$ since  no node in $R[k]$ satisfies $q$ (note that $k\!\not\in\!R[k]$). Concluding, it follows that $k_0\nvDash[p\rightarrow(q\rightarrow r)]\rightarrow[(p\& q)\rightarrow r]$.
\hfill $\maltese\hspace{-0.805em}\Box$ \end{proof}

Similarly, we provide an exact characterizations for Kripke models which satisfy the axiom $(A_5b)$.

\begin{theorem}[${\bf A_5b}$ \&  Transitivity+Persistency]\label{transitive2}
The axiom $(A_5b)$ is satisfied in every Kripke frame $\langle K,R \rangle$     in which  for all $k\!\in\!K$ the relation $R|_{R^{+}[k]}$  is transitive  and $\vDash|_{R^{++}[k]}$ is formula persistent with respect to $R$.
Conversely, if
 $(A_5b)$ is satisfied in a Kripke frame  $\langle K,R \rangle$     then  for all $k\!\in\!K$ the relation $R|_{R^{+}[k]}$  is transitive  and the restriction of the satisfaction relations to the sets $R^{++}[k]$ (for every $k\!\in\!K$) on that  frame  should be   formula  persistent with respect to $R$.
\end{theorem}
\begin{proof}
For a Kripke model $\langle K,R,\vDash\rangle$ and a node $k\!\in\!K$ of it, if $R|_{R^{+}[k]}$  is transitive and $\vDash|_{R^{++}[k]}$ is formula persistent with respect to $R$, then we show that $k\vDash(A_5b)$ which is equivalent to $\forall k'\!\in\!R[k]\big(k'\vDash(\varphi\&\psi)\rightarrow\theta
\Longrightarrow
k'\vDash\varphi\rightarrow(\psi\rightarrow\theta)\big)$  or equivalent to $\forall k''\!\in\!R[k']\,\big(k''\vDash\varphi\Longrightarrow
k''\vDash\psi\rightarrow\theta\big),$  under the assumption $kRk'\vDash(\varphi\&\psi)\rightarrow
 \theta $. This, in turn, is equivalent to   $\forall
k'''\!\in\!R[k'']\,\big(k'''\vDash\psi\Longrightarrow
k'''\vDash\theta\big)$ assuming that $k'Rk''\vDash\varphi$. Whence, we assume that (1)~the relation $\vDash|_{R^{++}[k]}$ is atom persistent with respect to $R$,
(2)~the restricted relation $R|_{R^{+}[k]}$ is transitive and we have that $kRk'Rk''Rk'''$,  (3)~$k'\vDash(\varphi\&\psi)\rightarrow\theta$,
(4)~$k''\vDash\varphi$ and (5)~$k'''\vDash\psi$; and show that $k'''\vDash\theta$: From (1), (4) and (5), noting that $k'',k'''\!\in\!R^{++}[k]$,   we have (6)~$k'''\vDash\varphi\&\psi$. Then from   (2) we have $k'Rk'''$ and so (3) and (6)  imply that $k'''\vDash\theta$ holds.

Now, if for a node $k_0\!\in\!K$
in a Kripke frame $\langle K,R\rangle$  the restricted relation
$R|_{R^{+}[k_0]}$ is not transitive, then there are $k_1, k_2, k_3\!\in\!R^{+}[k_0]$ such that $k_1Rk_2Rk_3$, but $k_1\!\!\!\not\!R k_3$. Also,  there are $\ell_1,\cdots,\ell_n\!\in\!K$ ($n\!\geqslant\!0$) such that $k_0R\ell_1R\cdots R\ell_nRk_1$. Let the satisfaction relation $\vDash$ be defined as $\big(R[k_1]\times \{r\}\big)\cup\{ \langle k_2, p\rangle\}\cup\{\langle k_3,q\rangle\}$ for some atoms $p,q,r$. Now we show that the instance $[(p\& q)\rightarrow
 r]\rightarrow[p\rightarrow(q\rightarrow r)]$ of $(A_5b)$ is not satisfied at $\ell_n$: we have $k_1\vDash p\& q\rightarrow r$ because for no $k\!\in\!R[k_1]$ can we have $k\vDash p\& q$. Also, $k_2\nvDash q\rightarrow r$ because at $k_3\!\in\!R[k_2]$ we have  $k_3\vDash q$ but $k_3\nvDash r$ (notice that $k_3\!\not\in\!R[k_1]$), therefore,
 $k_1\nvDash p\rightarrow(q\rightarrow r)$ because at $k_2\!\in\!R[k_1]$ we have $k_2\vDash p$ but $k_2\nvDash q\rightarrow r$. Now that we have $k_1\vDash p\& q\rightarrow r$ and $k_1\nvDash p\rightarrow(q\rightarrow r)$ we therefore infer the desired conclusion  $\ell_n\nvDash[(p\& q)\rightarrow
 r]\rightarrow[p\rightarrow(q\rightarrow r)]$.

Finally, if for a node $k_0\!\in\!K$
in a Kripke model $\langle K,R,\vDash \rangle$  the  restricted satisfaction relation
$\vDash|_{R^{++}[k_0]}$ is not formula persistent (with respect to $R$), then there exist two nodes $k_1,k_2\!\in\!R^{++}[k_0]$ and a formula $\varphi$ such that $k_1Rk_2$, $k_1\vDash\varphi$ and $k_2\nvDash\varphi$. Also, by Definition~\ref{def-rtclosed}, there   are $\ell_1,\cdots,\ell_n\!\in\!K$ ($n\!\geqslant\!1$) such that $k_0R\ell_1R\cdots R\ell_nRk_1$.
We show that the instance  $[(\varphi\& \overline{\top\!\!\!\!\!\!\top})\rightarrow\varphi]
\rightarrow[\varphi\rightarrow(\overline{\top\!\!\!\!\!\!\top}
 \rightarrow\varphi)]$  of $(A_5b)$ (see Remark~\ref{remark} for the definition of $\overline{\top\!\!\!\!\!\!\top}$) is not satisfied in this model at $\ell_{n-1}$; let us recall that if $n\!=\!1$ then $\ell_{n-1}\!=\!k_0$. To see this, firstly, we note that
for $\ell_n\!\in\!R[\ell_{n-1}]$ we have $\ell_n\vDash(\varphi\& \overline{\top\!\!\!\!\!\!\top})\rightarrow\varphi$ (indeed  $k\vDash(\varphi\& \overline{\top\!\!\!\!\!\!\top})\rightarrow\varphi$ holds for any node $k$). Secondly, $\ell_n\nvDash\varphi\rightarrow(\overline{\top\!\!\!\!\!\!\top}
 \rightarrow\varphi)$ because for $k_1\!\in\!R[\ell_n]$ we have $k_1\vDash\varphi$ but $k_1\nvDash \overline{\top\!\!\!\!\!\!\top}
 \rightarrow\varphi $ since at $k_2\!\in\!R[k_1]$ we have (of course $k_2\vDash\overline{\top\!\!\!\!\!\!\top}$ and also) $k_2\nvDash\varphi$.
\hfill $\maltese\hspace{-0.805em}\Box$ \end{proof}

Let us pause for a moment and see where we have got from these results so far. By Proposition~\ref{trivial} the axioms $(A_2)$, $(A_3)$ and $(A_7)$ (and also G\"odel's Axiom $\varphi\rightarrow\varphi\&\varphi$) are satisfied in {\em all} Kripke frames. By Theorem~\ref{reflexive1} only {\em reflexive} Kripke frames can satisfy the $(M\!P)$ rule. By Theorem~\ref{transitive1} the axiom $(A_1)$ can be satisfied in a  Kripke frame $\langle K,R \rangle$ if and only if $R|_{R^+[k]}$ is {\em transitive}, for all $k\!\in\!K$. So,  suitable Kripke frames for fuzzy logics should be {\em reflexive and transitive} by Lemma~\ref{lem-trans}. Moreover, the satisfaction relations on those (reflexive and transitive) Kripke frames should be (formula) {\em persistent} by Theorem~\ref{reflexivety2}, since Kripke models on those frames should satisfy the axiom $(A_4)$ as well; Theorem~\ref{reflexive2} (for the axiom $A_5a$) and Theorem~\ref{transitive2}  (for the axiom $A_5b$)  confirm this even more. So, one should necessarily  consider {\em reflexive, transitive and persistent} Kripke models for fuzzy logics.

Unfortunately, we have been unable to find a good characterizations for Kripke frame/models which satisfy the axiom $(A_6)$. One candidate for a class of Kripke frames which satisfy this axiom is the class of connected (Definition~\ref{def-conn}) Kripke frames. Indeed,  $(A_6)$ is satisfied in every ({\em persistent} and) {\em connected} Kripke model (see Theorem~\ref{connected} below). But the converse does not hold: the Kripke model $\langle {\big\{}\emptyset,\{a\},\{b\}{\big\}},\subseteq,\emptyset \rangle$ (with the empty satisfaction relation) is reflexive, transitive and  persistent but not connected (assuming $a\neq b$); while it satisfies   $(A_6)$, and every classical tautology. Below (in Theorem~\ref{connected}) we show that if a reflexive and transitive Kripke frame satisfies $(A_6)$ with persistent satisfaction relations, then it must be connected.

%\begin{example}
%The  Kripke model
%$\langle\{0,1,2,3\},\preceq',\vDash\rangle$, where
%$\preceq'=\leqslant -\{\langle 2,2\rangle\}$ and  for atoms $p,q,r$ the satisfaction relation  $\vDash$ is $\{\langle 3,p\rangle,\langle 3,q\rangle\}$, is
%non-reflexive (since we have $2\npreceq' 2$), transitive and persistent. Now, we have $2\vDash (q\rightarrow p)\rightarrow r$
%\end{example}

Before proving Theorem~\ref{connected} let us  make a little note about the linearity axiom $(\varphi\rightarrow\psi)\vee(\psi\rightarrow\varphi)$ which, over the (propositional) Intuitionistic Logic, axiomatizes the Kripke frames whose accessibility relations are linear orders. The  logic resulted  by appending this axiom to the intuitionistic logic is called Dummett logic (see~\cite{Dummett} and the Conclusions below).

\begin{lemma}[The Connectedness  Axiom]\label{persis} The formula
 $(\varphi\rightarrow\psi)\vee(\psi\rightarrow\varphi)$ is
 satisfied in all (formula) persistent and connected Kripke models.
\end{lemma}
\begin{proof}
For formulas $\varphi,\psi$,
 if $k\nvDash(\varphi\rightarrow\psi)\vee(\psi\rightarrow\varphi)$ then there exist $k',k''\!\in\!R[k]$ such that $k'\vDash\varphi$
 but $k'\nvDash\psi$, and $k''\vDash\psi$ but $k''\nvDash\varphi$. By connectedness (and $k',k''\!\in\!R^{+}[k]$)  we have either $k' R k''$ or $k'' R k'$.
  Now, if  $k' R k''$ then from $k'\vDash \varphi$ we will have $k''\vDash\varphi$ by (formula) persistency; a contradiction (since $k''\nvDash\varphi$). Similarly, a contradiction follows from $k'' R k'$.
\hfill $\maltese\hspace{-0.805em}\Box$ \end{proof}

\begin{theorem}[$A_6$ \& Connectedness, by  \textsl{Reflexivity, Transitivity and Persistency}]\label{connected}
The axiom $(A_6)$ is satisfied in every connected and persistent Kripke model. Also, if a reflexive and transitive Kripke frame satisfies $(A_6)$ with persistent satisfaction relations, then it must be connected.
\end{theorem}

\begin{proof}
Suppose $\langle K,R,\vDash \rangle$ is connected and persistent. For a node $k\!\in\!K$, and formulas $\varphi,\psi,\theta$, we show that $k\vDash[(\varphi\rightarrow\psi)\rightarrow\theta]\rightarrow
[([\psi\rightarrow\varphi]\rightarrow\theta)\rightarrow\theta]$. This is equivalent to $\forall k'\!\in\! R[k]
\big(k'\vDash(\varphi\rightarrow\psi)
\rightarrow\theta\Longrightarrow
k'\vDash([\psi\rightarrow\varphi]
\rightarrow\theta)\rightarrow\theta\big)$. So,   fix a    $k'\!\in\!R[k]$  with  $k'\vDash(\varphi\rightarrow\psi)\rightarrow\theta$; we prove that $\forall k''\!\in\! R[k']\,\big(k''\vDash[(\psi\rightarrow\varphi)
\rightarrow\theta]\Longrightarrow k''\vDash\theta\big)$. Whence, we assume that
(1)~$R$  is  connected and $k R k' R k''$,
 (2)~$\vDash$ is formula persistent with respect to $R$,
(3)~$k'\vDash(\varphi\rightarrow\psi)\rightarrow\theta$, and
 (4)~$k''\vDash(\psi\rightarrow\varphi)\rightarrow\theta$;  and show that $k''\vDash\theta$. By Lemma~\ref{persis} we have either (i)~$k''\vDash\varphi\rightarrow\psi$ or (ii)~$k''\vDash\psi\rightarrow\varphi$. In case of (i), from (1) and (3) we already infer that $k''\vDash\theta$. In case of (ii), we  note that $k''Rk''$ by~(1) (and that the connectedness of $R$ implies the reflexivity of $R|_{R^{+}[k]}$) and so from (4) we can conclude that $k''\vDash\theta$.

Now, assume (for the sake of contradiction) that the Kripke frame $\langle K,R \rangle$ is reflexive and transitive but not connected. Then there must exist some nodes $k,k',k''\!\in\!K$ such that $kRk'$, $kRk''$, $k'\!\!\!\!\not\!Rk''$ and $k''\!\!\!\!\not\!Rk'$. Let us already note that then $k\!\not\in\!R[k']\cup R[k'']$ and $k'\!\not\in\!R[k'']$ also $k''\!\not\in\!R[k']$. For atoms $p,q,r$, define the satisfaction relation $\vDash$ on this frame to be
$(R[k']\times\{p\})\cup(R[k'']\times\{q\})
\cup([R[k]\cap\{\ell\!\in\!K\mid\ell\!\!\!\not\!Rk\}]\times\{r\})$.
By the transitivity of $R$, this satisfaction relation is atom persistent (since, e.g., if $\ell\vDash r$ and $\ell R\ell'$ then from $kR\ell$ and $\ell\!\!\!\not\!Rk$, and the transitivity of $R$, we have $kR\ell'$ and also $\ell'\!\!\!\!\not\!Rk$ since otherwise if $\ell'Rk$ then from $\ell R\ell'$, and the transitivity of $R$, we would have $\ell Rk$ contradiction); thus $\vDash$ is  formula persistent  (by the transitivity of $R$ and Proposition~\ref{lem-per}). We show that under this satisfaction relation    the instance $[(p\rightarrow q)\rightarrow r]\rightarrow[([q\rightarrow p]\rightarrow r)\rightarrow
r]$ of $(A_6)$ is not satisfied at $k$. We firstly  note that $k\nvDash p\rightarrow q$ (because at $k'\!\in\!R[k]$ we have $k'\vDash p$ and $k'\nvDash q$) and also $k\nvDash q\rightarrow p$ (because at $k''\!\in\!R[k]$ we have $k''\vDash q$ and $k''\nvDash p$), and secondly that $k\vDash (p\rightarrow q)\rightarrow r$ and  $k\vDash (q\rightarrow p)\rightarrow r$ (because for any $\ell\!\in\!R[k]$ if $\ell\vDash p\rightarrow q$ or $\ell\vDash q\rightarrow p$ then, by the persistency,  $\ell\!\!\!\not\!Rk$ and so $\ell\vDash r$). Finally, $k\nvDash ([q\rightarrow p]\rightarrow r)\rightarrow r$ since $k\vDash [q\rightarrow p]\rightarrow r$ but $k\nvDash r$.
\hfill $\maltese\hspace{-0.805em}\Box$ \end{proof}

Finally, the main result of the paper is the following which  follows from all the previous results:

 \begin{corollary}[Kripke Models for the Basic Fuzzy Logic]\label{theorem}
A Kripke model satisfies the axioms (and the rule) of the Basic Fuzzy
Logic   if and only if it is reflexive, transitive, and connected, and the satisfaction relation is (formula) persistent with respect to the accessibility relation.
\hfill $\maltese\hspace{-0.805em}\Box$
\end{corollary}

This can indeed be seen as a negative result in the theory of Kripke models, since it shows that no class of Kripke frames can axiomatize exactly BL or the fuzzy logics that do not contain G\"odel logic. But it has also some positive sides discussed in the next section.

%%%%%%%%%%%%%%%%%%%%%%%%%%%%%%
%%%%%%%%%%%%%%%%%%%%%%%%%%%%%%
%%%%%%%

\section{Conclusions}

G\"odel Fuzzy Logic is axiomatized as BL plus the axiom
$\varphi\rightarrow(\varphi\&\varphi)$ of idempotence of
conjunction (cf.~\cite{Bendova}). Dummett~\cite{Dummett} showed
that  this logic can be completely axiomatized by the axioms of
intuitionistic logic plus the axiom
$(\varphi\rightarrow\psi)\vee(\psi\rightarrow\varphi)$. Indeed,
the G\"odel--Dummett Logic is sound and strongly complete with
respect to reflexive, transitive,   connected and persistent   Kripke models. In
Corollary~\ref{theorem}, we showed that the only class of Kripke
models  which could be sound and (strongly) complete for a logic
containing BL must contain the class of reflexive, transitive,
connected and persistent  Kripke models. In the other words, any logic that contains
BL and is axiomatizing a class of Kripke frames/models must also contain the
G\"odel--Dummett Logic (cf. Proposition~\ref{trivial}). So, a
Kripke-Model-Theoretic characterization of G\"odel Fuzzy Logic is
that \emph{it is the smallest fuzzy logic containing the Basic Fuzzy  Logic
which is sound and complete with respect to a class of Kripke
frames/models}. Also, the class of reflexive, transitive, connected and persistent
Kripke models is the smallest class that can be axiomatized by a
propositional fuzzy logic.

%\begin{figure}
%  \includegraphics{example.eps}
%\end{figure}
%\begin{figure*}
%  \includegraphics[width=0.75\textwidth]{example.eps}
%\end{figure*}
%\begin{table}
%\begin{tabular}{lll}
%\hline\noalign{\smallskip}
%\end{tabular}
%\end{table}



\end{document}